\documentclass[a4paper,10pt]{amsart}
\usepackage{etex}

\usepackage[T1]{fontenc}
\usepackage{lmodern}
\usepackage[utf8]{inputenc}
\usepackage{amssymb}
\usepackage{enumitem}
\usepackage{mathrsfs,dsfont,mathtools}
\usepackage{nicefrac}

\usepackage{xcolor}

\usepackage{tikz}
\usetikzlibrary{matrix,arrows,calc}
\tikzset{cd/.style=matrix of math nodes,row sep=2em,column sep=2em, text height=1.5ex, text depth=0.5ex}
\tikzset{cdar/.style=->,auto}
\tikzset{mid/.style={anchor=mid}}

\usepackage[pdftitle={Extensions of Cuntz-Krieger algebras},
  pdfauthor={Rasmus Bentmann},
  pdfsubject={Mathematics}
]{hyperref}
\usepackage[lite]{amsrefs}
\newcommand*{\MRref}[2]{ \href{http://www.ams.org/mathscinet-getitem?mr=#1}{MR \textbf{#1}}}
\newcommand*{\arxiv}[1]{\href{http://www.arxiv.org/abs/#1}{arXiv: #1}}

\renewcommand{\PrintDOI}[1]{\href{http://dx.doi.org/\detokenize{#1}}{doi: \detokenize{#1}}%
  \IfEmptyBibField{pages}{, (to appear in print)}{}}

\usepackage{microtype}

\usepackage[all]{xy}

\theoremstyle{plain}
\newtheorem{theorem}[equation]{Theorem}
\newtheorem{lemma}[equation]{Lemma}

\newtheorem{corollary}[equation]{Corollary}
\theoremstyle{definition}
\newtheorem{definition}[equation]{Definition}

\theoremstyle{remark}
\newtheorem{remark}[equation]{Remark}

\newcommand*{\Z}{\mathbb Z}

\newcommand*{\trans}{\textup{t}}
\newcommand*{\FK}{\textup{FK}}
\newcommand*{\KK}{\textup{KK}}

\newcommand*{\K}{\textup{K}}

\newcommand*{\Cstar}{\texorpdfstring{$\textup C^*$\nb-}{C*-}}

\newcommand*{\UNIT}{\mathds{1}}

\newcommand*{\nb}{\nobreakdash}  

\mathtoolsset{mathic}

\DeclareMathOperator{\Ext}{Ext}

\DeclareMathOperator{\Coker}{Coker}
\DeclareMathOperator{\Prim}{Prim}

\DeclareMathOperator{\Ker}{Ker}

\newcommand*{\KKcat}{\mathfrak{KK}}

\newcommand*{\defeq}{\mathrel{\vcentcolon=}}

\newcommand{\Asf}{\mathsf{A}}

\newcommand{\Nsf}{\mathsf{N}}

\newcommand*{\Cuntz}{\mathcal O}
\newcommand*{\Acal}{\mathcal A}
\newcommand*{\Bcal}{\mathcal B}
\newcommand*{\Ecal}{\mathcal E}
\newcommand*{\Comp}{\mathcal K}
\newcommand*{\Ical}{\mathcal I}

\newdir{ >}{{}*!/-5pt/@{>}}

\begin{document}

\title{Extensions of Cuntz--Krieger algebras}

\author{Rasmus Bentmann}
\email{rasmus.bentmann@mathematik.uni-goettingen.de}

\address{Mathematisches Institut\\
  Georg-August Universit\"at G\"ottingen\\
  Bunsenstra\ss{}e 3--5\\
  37073 G\"ottingen\\
  Germany}

\begin{abstract}
We show that a \Cstar{}algebra ``looking like'' a Cuntz--Krieger algebra is a Cuntz--Krieger algebra.  This implies that, in an appropriate sense, the class of Cuntz--Krieger algebras is closed under extensions of real rank zero.
\end{abstract}

\thanks{This work was supported by the German Research Foundation (DFG) through the project grant ``Classification of non-simple purely infinite \Cstar{}algebras''}



\maketitle

\section{Introduction and statement of results}
\label{sec:intro}

The Cuntz--Krieger algebras introduced around 1980 in~\cites{Cuntz-Krieger:topological_Markov_chains, Cuntz:topological_Markov_chains_II} form a class of purely infinite \Cstar{}algebras (not necessarily simple) with close ties to symbolic dynamics containing Cuntz's algebras~$\Cuntz_n$ from~\cite{Cuntz:O_n}.

It was shown in~\cites{Cuntz:topological_Markov_chains_II} that a quotient of a Cuntz--Krieger algebra is again a Cuntz--Krieger algebra and that an ideal in a Cuntz--Krieger algebra is stably isomorphic to a Cuntz--Krieger algebra. The original works~\cites{Cuntz-Krieger:topological_Markov_chains, Cuntz:topological_Markov_chains_II} also contain computations of $\Ext$\nb-groups of Cuntz--Krieger algebras. However, the following fundamental question has remained open: \emph{when is an extension of Cuntz--Krieger algebras again a Cuntz--Krieger algebra?} (Or, to be more precise, when is an extension of a Cuntz--Krieger algebra by a \Cstar{}algebra stably isomorphic to a Cuntz--Krieger algebra again a Cuntz--Krieger algebra?) We shall see that this is the case if and only if the extension is unital and has real rank zero.

In the light of Kirchberg's classification theorem for certain non-simple purely infinite \Cstar{}algebras~\cite{Kirchberg:Michael}, one may even feel encouraged to ask: \emph{which unital \Cstar{}algebras are Cuntz--Krieger algebras?} In fact, once the simple case is settled, this question is basically equivalent to the first one. We shall see that it is indeed possible to characterize the class of Cuntz--Krieger algebras solely by \Cstar{}algebraic properties, as suggested by the following definition.

\begin{definition}[confer \cite{Arklint:PhantomCK}*{Definition~1.1}]
  \label{def:looks_like}
A \Cstar{}algebra $\Acal$ \emph{looks stably like a Cuntz--Krieger algebra} if $\Acal$ is separable, nuclear and $\Cuntz_\infty$\nb-absorbing, $\Acal$~has real rank zero and finitely many ideals, and every simple subquotient of~$\Acal$ satisfies the universal coefficient theorem and has a finitely generated, free $\K_1$\nb-group of the same rank as its $\K_0$\nb-group.  A \Cstar{}algebra $\Acal$ \emph{looks like a Cuntz--Krieger algebra} if it is unital and looks stably like a Cuntz--Krieger algebra.
\end{definition}

As the definition indicates, we only consider Cuntz--Krieger algebras of real rank zero, that is, Cuntz--Krieger algebras~$\Cuntz_\Asf$ where $\Asf$ is a $\{0,1\}$\nb-valued square matrix~$\Asf$ satisfying Cuntz's Condition~\textup{(II)} (see \cite{Cuntz:topological_Markov_chains_II}*{Page 29}; Cuntz--Krieger algebras of non-zero matrices that do not satisfy Condition~(II) have uncountably many ideals and positive real rank). Notice that every Cuntz--Krieger algebra looks like a Cuntz--Krieger algebra. We prove the converse below.

In the terminology of \cite{Arklint:PhantomCK}, the theorem says that \emph{no phantom Cuntz--Krieger algebras exist}. This generalizes several results established in \cites{Arklint:PhantomCK,Arklint-Bentmann-Katsura:Reduction,Arklint-Bentmann-Katsura:Range,Bentmann:Real_rank_zero_and_int_cancellation,Eilers-Katsura-Tomforde-West:Range_nonsimple_graph}, where the same conclusion was reached under various additional assumptions on the ideal lattice or the $\K$\nb-theory of the \Cstar{}algebra~$\Acal$. These partial solutions to the problem employed complete classification invariants for purely infinite nuclear \Cstar{}algebras in the considered class together with a description of the range of the invariant on the Cuntz--Krieger algebras in the class. In the present paper, the strategy is rather to inductively construct a $\KK(X)$-equivalent Cuntz--Krieger algebra without using an invariant.

\begin{theorem}
  \label{thm:no_phantoms}
Let $\Acal$ be a \Cstar{}algebra that looks \textup{(}stably\textup{)} like a Cuntz--Krieger algebra.  Then $\Acal$ is \textup{(}stably\textup{)} isomorphic to a Cuntz--Krieger algebra.
\end{theorem}

The theorem is proved in~\S\ref{sec:proof}. The desired result about extensions is now a simple consequence.

\begin{corollary}
Let $\Ical$ be stably isomorphic to a Cuntz--Krieger algebra and let $\Bcal$ be a Cuntz--Krieger algebra. Let $\Ical\rightarrowtail\Acal\twoheadrightarrow\Bcal$ be an extension of \Cstar{}algebras such that $\Acal$~is unital. The following conditions are equivalent:
 \begin{enumerate}[label=\textup{(\arabic{*})}]
  \item $\Acal$ is a Cuntz--Krieger algebra;
  \item $\Acal$ has real rank zero;
  \item the exponential map $\K_0(\Bcal)\to\K_1(\Ical)$ of the extension vanishes.
 \end{enumerate}
\end{corollary}

\begin{proof}
As an extension of Cuntz--Krieger algebras, $\Acal$ is again separable, nuclear and $\Cuntz_\infty$\nb-absorbing (by \cite{Toms-Winter:Strongly_self-abs}*{Theorem~4.3}). Hence (1)$\Leftrightarrow$(2) follows from Theorem~\ref{thm:no_phantoms} and (2)$\Leftrightarrow$(3) follows from \cite{Pasnicu-Rordam:Purely_inf_RR0}*{Theorem~4.2}.
\end{proof}

\emph{A posteriori} and based on Restorff's classification theorem, we also obtain an ``ex\-ter\-nal'' classification result for Cuntz--Krieger algebras. By~$\FK$ we denote \emph{reduced filtered $\K$\nb-theory} as defined in~\cite{Restorff:Classification}*{Definition~4.1}.

\begin{corollary}
Let $\Acal$ be \textup{(}stably isomorphic to\textup{)} a Cuntz--Krieger algebra and let $\Bcal$ be a separable nuclear $\Cuntz_\infty$\nb-absorbing \Cstar{}algebra such that $\Prim(\Acal)\cong\Prim(\Bcal)$ and $\FK(\Acal)\cong\FK(\Bcal)$. Assume that every simple subquotient of~$\Bcal$ satisfies the universal coefficient theorem. Then $\Acal$ is stably isomorphic to~$\Bcal$.
\end{corollary}

\begin{proof}
\cite{Pasnicu-Rordam:Purely_inf_RR0}*{Theorem~4.2} shows that $\Acal$ is $\K_0$\nb-liftable in the sense of \cite{Pasnicu-Rordam:Purely_inf_RR0}*{Definition~3.1}. Hence $\Bcal$ is $\K_0$\nb-liftable because of $\FK(\Acal)\cong\FK(\Bcal)$ and \cite{Bentmann:Real_rank_zero_and_int_cancellation}*{Corollary~3.6}. Thus $\Bcal$ has real rank zero by \cite{Pasnicu-Rordam:Purely_inf_RR0}*{Theorem~4.2}. This shows that $\Bcal$ looks stably like a Cuntz--Krieger algebra. Theorem~\ref{thm:no_phantoms} implies that $\Bcal$ is stably isomorphic to a Cuntz--Krieger algebra. Thus $\Acal$ is stably isomorphic to~$\Bcal$ by~\cite{Restorff:Classification}*{Theorem~4.2}.
\end{proof}

Cuntz--Krieger algebras are a special case of graph \Cstar{}algebras (see~\cite{Raeburn:Graph_algebras}). The extent to which the methods and results in this article generalize to (purely infinite) graph \Cstar{}algebras shall be determined elsewhere.

\subsection*{Acknowledgement}
I am grateful to S\o{}ren Eilers for promoting the questions answered in this article and to Ralf Meyer for valuable discussions.

\section{Proof of Theorem~\ref{thm:no_phantoms}}
\label{sec:proof}

\begin{remark}
The proof makes use of the formalism of \Cstar{}algebras over topological spaces from~\cite{Meyer-Nest:Bootstrap}*{\S2} and a version of Kasparov theory for these, including its structure as a triangulated category~\cite{Meyer-Nest:Bootstrap}*{\S3}; we adopt the notation from~\cite{Meyer-Nest:Bootstrap}, in particular for the extension functor~$i_Y^X$, the restriction functor~$r_X^Y$ and the projection functor~$P_Y\defeq i_Y^X\circ r_X^Y$ for a locally closed subset~$Y$ of a space~$X$ introduced in~\cite{Meyer-Nest:Bootstrap}*{Definitions~2.18, 2.19 and~4.2}.

We will use various results from~\cites{Cuntz-Krieger:topological_Markov_chains, Cuntz:topological_Markov_chains_II}. There are two points we need to address to ascertain the applicability of these results in our setting.

Firstly, it will be more convenient for us to use matrices~$\Asf$ with non-negative entries (rather than entries in $\{0,1\}$) satisfying Condtion~(K) (the appropriate generalization of Condtion~(II); see~\cite{Kumjian-Pask-Raeburn:Cuntz-Krieger_graphs}). The \Cstar{}algebra~$\Cuntz_\Asf$ is then still defined, and the resulting class of \Cstar{}algebras remains the same. Indeed, one can turn every non-negative integer-valued matrix~$\Asf$ into a $\{0,1\}$-matrix without changing the isomorphism class of the associated \Cstar{}algebra using a sequence of so-called out-split moves (see \cite{Bates-Pask:Flow_eq_for_graph_algs}*{Theorem~3.2}) or by replacing~$\Asf$ with the edge matrix of a graph whose adjacency matrix is~$\Asf$ (see \cite{Kumjian-Pask-Raeburn-Renault:Graphs}*{\S4}). This observation also provides a receipt for checking that a given result proved for Cuntz--Krieger algebras of $\{0,1\}$-matrices generalizes to Cuntz--Krieger algebras of non-negative matrices.

Secondly, on \cite{Cuntz:topological_Markov_chains_II}*{Page~38}, an assumption on~$\Asf$ is introduced which guarantees that $\Cuntz_\Asf$ has a unique non-zero proper ideal, so that the pair $\left[\Ecal_{\Nsf}\right]$, $\Ext(\Cuntz_{\Asf_1})\otimes\K_0(\Cuntz_{\Asf_2})$ is an invariant of the plain \Cstar{}algebra~$\Cuntz_{\Asf}$ because the extension $\Comp\otimes\Cuntz_{\Asf_2}\rightarrowtail\Cuntz_{\Asf}\twoheadrightarrow\Cuntz_{\Asf_1}$ is uniquely determined. However, without this assumption, the computations in the proof of \cite{Cuntz:topological_Markov_chains_II}*{Proposition~4.4} still go through with minor notational modifications due to the fact that the ideal need no longer be stable.
\end{remark}

Let us now begin with the proof of Theorem~\ref{thm:no_phantoms}. We proceed by induction on the number~$n$ of primitive ideals in~$\Acal$. The result for $n=1$ follows from a version of Szyma\'nski's theorem \cite{Eilers-Katsura-Tomforde-West:Range_nonsimple_graph}*{Proposition~3.9} and the Kirchberg--Phillips classification theorem \cite{Rordam-Stormer:Classification_Entropy}*{Theorem~8.4.1}. We may thus assume that $n>1$ and that the theorem holds for \Cstar{}algebras with fewer than~$n$ primitive ideals.

Let $\Acal$ look stably like a Cuntz--Krieger algebra. We may assume that $X\defeq\Prim(\Acal)$ is connected because otherwise $\Acal$ is the direct sum of two proper ideals. Let $x\in X$ be a point such that the subset~$\{x\}\subseteq X$ is closed. We apply the theorem to the restrictions~$r_X^{\{x\}} A$ and~$r_X^{X\setminus\{x\}} A$ and obtain square matrices~$\Asf_1$ and~$\Asf_2$ such that $\Prim(\Cuntz_{\Asf_1})\cong\{x\}$ and $\Prim(\Cuntz_{\Asf_2})\cong X\setminus\{x\}$ and such that $r_X^{\{x\}}\Acal\simeq\Cuntz_{\Asf_1}$ in $\KKcat(\{x\})$ and $r_X^{X\setminus\{x\}}\Acal\simeq\Cuntz_{\Asf_2}$ in $\KKcat(X\setminus\{x\})$. Moreover, it is ensured by \cite{Eilers-Katsura-Tomforde-West:Range_nonsimple_graph}*{Proposition~3.9} that we may choose $\Asf_1$ such that each diagonal entry is greater than~$1$. The adjunctions in \cite{Meyer-Nest:Bootstrap}*{Proposition~3.12} yield $\KK(X)$\nb-equivalences $P_{\{x\}}\Acal\simeq i_{\{x\}}^X\Cuntz_{\Asf_1}$ and $P_{X\setminus\{x\}}\Acal\simeq i_{X\setminus\{x\}}^X\Cuntz_{\Asf_2}$.

Since the \Cstar{}algebra~$\Acal(X)$ is nuclear, \cite{Meyer-Nest:Bootstrap}*{Proposition~4.10} shows that $\Acal$ belongs to the subclass $\KKcat(X)_\mathrm{loc}\subseteq\KKcat(X)$ defined in \cite{Meyer-Nest:Bootstrap}*{Definition~4.8}. Hence \cite{Meyer-Nest:Bootstrap}*{Proposition~4.7} implies that the extension $P_{X\setminus\{x\}}\Acal\rightarrowtail\Acal\twoheadrightarrow P_{\{x\}}\Acal$ of \Cstar{}algebras over~$X$ is admissible in the sense of \cite{Meyer-Nest:Bootstrap}*{Definition~3.10} and therefore induces an exact triangle in the category~$\KKcat(X)$.

Consider the following diagram in the category~$\KKcat(X)$:
\begin{equation}
\label{eq:morphism_of_triangles}
  \begin{tikzpicture}[baseline=(current bounding box.west)]
    \matrix(m)[cd,column sep=4em]{
      \Sigma P_{\{x\}}\Acal&P_{X\setminus\{x\}}\Acal&\Acal&P_{\{x\}}\Acal\\
      \Sigma i_{\{x\}}^X\Cuntz_{\Asf_1}&i_{X\setminus\{x\}}^X\Cuntz_{\Asf_2}& &i_{\{x\}}^X\Cuntz_{\Asf_1}.\\
    };
    \begin{scope}[cdar]
      \draw (m-1-1) -- (m-1-2);
      \draw (m-1-2) -- (m-1-3);
      \draw (m-1-3) -- (m-1-4);
      \draw (m-2-1) -- node {\(\alpha\)} (m-2-2);
      \draw (m-1-1) -- node {\(\simeq\)} (m-2-1);
      \draw (m-1-2) -- node {\(\simeq\)}(m-2-2);
      \draw (m-1-4) -- node {\(\simeq\)}(m-2-4);
    \end{scope}
  \end{tikzpicture}
\end{equation}
The upper row is the exact triangle induced by the extension $P_{X\setminus\{x\}}\Acal\rightarrowtail\Acal\twoheadrightarrow P_{\{x\}}\Acal$. The vertical arrows are arbitrarily chosen $\KK(X)$\nb-equivalences such that the leftmost arrow is the suspension of the rightmost one. The morphism~$\alpha$ is then chosen uniquely such that the left-hand square commutes; we wish to realize it by an extension over~$X$ whose underlying \Cstar{}algebra is a Cuntz--Krieger algebra. By the adjunction \cite{Meyer-Nest:Bootstrap}*{Proposition~3.13}, we have
\begin{multline*}
 \alpha\in\KK_0(X;\Sigma i_{\{x\}}^X\Cuntz_{\Asf_1},i_{X\setminus\{x\}}^X\Cuntz_{\Asf_2})
 \cong\KK_1(X;i_{\{x\}}^X\Cuntz_{\Asf_1},i_{X\setminus\{x\}}^X\Cuntz_{\Asf_2})\\
 \cong\KK_1\bigl(\Cuntz_{\Asf_1},(i_{X\setminus\{x\}}^X\Cuntz_{\Asf_2})(U_x)\bigr)
 \cong\KK_1\bigl(\Cuntz_{\Asf_1},\Cuntz_{\Asf_2}(U_x\setminus\{x\})\bigr)
\end{multline*}
where $U_x$ denotes the smallest open neighbourhood of~$x$ in~$X$. The ideal $\Cuntz_{\Asf_2}(U_x\setminus\{x\})$ of $\Cuntz_{\Asf_2}$ is isomorphic to $\Comp\otimes\Cuntz_{\Asf_2'}$ for an appropriately chosen submatrix~$\Asf_2'$ of~$\Asf_2$ by \cite{Cuntz:topological_Markov_chains_II}*{Theorem~2.5} and~\cite{Hjelmborg:Purely_inf_stable}*{Theorem~2.14}. Hence $\alpha$ corresponds to an element $\beta\in\KK_1\bigl(\Cuntz_{\Asf_1},\Cuntz_{\Asf_2'})\cong\Ext\bigl(\Cuntz_{\Asf_1},\Cuntz_{\Asf_2'})$ (see \cite{Blackadar:K-theory}*{Proposition 17.6.5}).

Since~$\Acal$ has real rank zero, the exponential map $\beta_*\colon\K_0(\Cuntz_{\Asf_1})\to\K_1(\Cuntz_{\Asf_2'})$ vanishes. By \cite{Cuntz:topological_Markov_chains_II}*{Theorem~3.11}, the class~$\beta$ belongs to a canonical subgroup of $\Ext\bigl(\Cuntz_{\Asf_1},\Cuntz_{\Asf_2'})$ isomorphic to $\Ext(\Cuntz_{\Asf_1})\otimes\K_0(\Cuntz_{\Asf_2'})$. We have $\Ext(\Cuntz_{\Asf_1})\cong\Coker(\Asf_1-\UNIT)$ by \cite{Cuntz-Krieger:topological_Markov_chains}*{Theorem~5.3} and $\K_0(\Cuntz_{\Asf_2'})\cong\Coker({\Asf_2'}^\trans-\UNIT)$ by \cite{Cuntz:topological_Markov_chains_II}*{Proposition~3.1}.

Let $\Nsf'\neq 0$ be a non-negative integer-valued matrix fitting into the block matrix
\[
 \Asf'\defeq\begin{pmatrix}\Asf_1&\Nsf'\\0&\Asf_2'\end{pmatrix}.
\]
This determines an extension of \Cstar{}algebras $\Ecal_{\Nsf'}=\left(\Comp\otimes\Cuntz_{\Asf_2'}\rightarrowtail\Cuntz_{\Asf'}\twoheadrightarrow\Cuntz_{\Asf_1}\right)$ by \cite{Cuntz:topological_Markov_chains_II}*{Theorem~2.5} and~\cite{Hjelmborg:Purely_inf_stable}*{Theorem~2.14}. If the size of~$\Nsf'$ is $m_1\times m_2'$ then \cite{Cuntz:topological_Markov_chains_II}*{Proposition~4.4} says that the class of~$\Ecal_{\Nsf'}$ in
\[
 \Coker(\Asf_1-\UNIT)\otimes\Coker({\Asf_2'}^\trans-\UNIT)
 \cong\Ext(\Cuntz_{\Asf_1})\otimes\K_0(\Cuntz_{\Asf_2'})
 \subseteq\Ext\bigl(\Cuntz_{\Asf_1},\Cuntz_{\Asf_2'})
\]
is given by regarding $\Nsf'$ as an element in $\Z^{m_1}\otimes\Z^{m_2'}$ in the obvious way and then applying the projection map $q\colon\Z^{m_1}\otimes\Z^{m_2'}\twoheadrightarrow\Coker(\Asf_1-\UNIT)\otimes\Coker({\Asf_2'}^\trans-\UNIT)$. Since we have arranged for the diagonal entries of~$\Asf_1$ to be at least~$2$, there exists a vector in the image of $\Asf_1-\UNIT$ with only positive entries. Hence $\Ker(q)$ contains a matrix with only positive entries. We may therefore choose a positive integer-valued matrix~$\Nsf'$ as above such that $\left[\Ecal_{\Nsf'}\right]=\beta\in\Ext\bigl(\Cuntz_{\Asf_1},\Cuntz_{\Asf_2'})$.

Now we define the matrix
\[
 \Asf\defeq\begin{pmatrix}\Asf_1&\Nsf\\0&\Asf_2\end{pmatrix},
\]
where $\Nsf$ is such that $\begin{pmatrix}\Nsf'\\\Asf_2'\end{pmatrix}$ is a submatrix of $\begin{pmatrix}\Nsf\\\Asf_2\end{pmatrix}$ and $\Nsf$ contains only vanishing columns besides the columns of~$\Nsf'$. Since the matrix~$\Nsf'$ is positive-valued, it follows from \cite{Cuntz:topological_Markov_chains_II}*{Theorem~2.5} that $\Prim(\Cuntz_\Asf)\cong X$. We obtain the extension
\[
\Ecal_\Nsf=\bigl(P_{X\setminus\{x\}}\Cuntz_{\Asf}\rightarrowtail\Cuntz_{\Asf}\twoheadrightarrow P_{\{x\}}\Cuntz_{\Asf}\bigr)
\]
of \Cstar{}algebras over~$X$. The quotient $P_{\{x\}}\Cuntz_{\Asf}$ is isomorphic over~$X$ to $i_{\{x\}}^X\Cuntz_{\Asf_1}$ because of $r_X^{\{x\}}\Cuntz_{\Asf}\cong\Cuntz_{\Asf_1}$ and \cite{Meyer-Nest:Bootstrap}*{Lemma~2.20}. Similarly, $P_{X\setminus\{x\}}\Cuntz_{\Asf}$ is stably isomorphic over~$X$ to $i_{X\setminus\{x\}}^X\Cuntz_{\Asf_2}$ because $r_X^{X\setminus\{x\}}\Cuntz_{\Asf}$ is stably isomorphic over~$X\setminus\{x\}$ to~$\Cuntz_{\Asf_2}$. To see the latter, recall from the proof of \cite{Cuntz:topological_Markov_chains_II}*{Theorem~2.5} that $\Cuntz_{\Asf_2}$ is a full corner in $\Cuntz_\Asf(X\setminus\{x\})$; the inclusion map provides a stable isomorphism over~$X\setminus\{x\}$. 

\begin{lemma}
Under the isomorphism 
\[
\KK_1(X;P_{\{x\}}\Cuntz_{\Asf},P_{X\setminus\{x\}}\Cuntz_{\Asf})
\cong\KK_1(X;i_{\{x\}}^X\Cuntz_{\Asf_1},i_{X\setminus\{x\}}^X\Cuntz_{\Asf_2})
\]
resulting from the idenfications above, the classes~$\left[\Ecal_\Nsf\right]$ and~$\alpha$ correspond to each other.
\end{lemma}

\begin{proof}
We have the following commutative diagram of group isomorphisms:
\begin{equation*}
  \begin{tikzpicture}[baseline=(current bounding box.west)]
    \matrix(m)[cd,column sep=4em]{
      \left[\Ecal_\Nsf\right]\in\KK_1(X;P_{\{x\}}\Cuntz_{\Asf},P_{X\setminus\{x\}}\Cuntz_{\Asf})&
      \KK_1(X;i_{\{x\}}^X\Cuntz_{\Asf_1},i_{X\setminus\{x\}}^X\Cuntz_{\Asf_2})\ni\alpha\\
      \KK_1\bigl(\Cuntz_{\Asf}(\{x\}),\Cuntz_{\Asf}(U_x\setminus\{x\})&
      \KK_1\bigl(\Cuntz_{\Asf_1},\Cuntz_{\Asf_2}(U_x\setminus\{x\})\bigl)\\
      \KK_1\bigl(\Cuntz_{\Asf'}(\{x\}),\Cuntz_{\Asf'}(U_x\setminus\{x\})\bigl)&
      \KK_1(\Cuntz_{\Asf_1},\Cuntz_{\Asf_2'})\ni\beta.\\
    };
    \begin{scope}[cdar]
      \draw (m-1-2) -- (m-1-1);
      \draw (m-2-2) -- (m-2-1);
      \draw (m-3-2) -- (m-3-1);
      \draw (m-1-1) -- (m-2-1);
      \draw (m-1-2) -- (m-2-2);
      \draw (m-3-1) -- (m-2-1);
      \draw (m-3-2) -- (m-2-2);
    \end{scope}
  \end{tikzpicture}
\end{equation*}
In the top square, the vertical arrows come from the adjunction \cite{Meyer-Nest:Bootstrap}*{Proposition~3.13}, and the horizontal arrows are induced by the isomorphism $\Cuntz_\Asf(\{x\})\cong\Cuntz_{\Asf_1}$ and the full corner embedding $\Cuntz_{\Asf_2}\hookrightarrow\Cuntz_\Asf(X\setminus\{x\})$; the square commutes by naturality.

The bottom square is induced by the commuting square of corner embeddings
\begin{equation*}
  \begin{tikzpicture}[baseline=(current bounding box.west)]
    \matrix(m)[cd,column sep=4em]{
      \Cuntz_{\Asf}&
      \Cuntz_{\Asf_2}\\
      \Cuntz_{\Asf'}&
      \Cuntz_{\Asf_2'}\\
    };
    \begin{scope}[cdar]
      \draw[>->] (m-1-2) -- (m-1-1);
      \draw[>->] (m-2-2) -- (m-2-1);
      \draw[>->] (m-2-1) -- (m-1-1);
      \draw[>->] (m-2-2) -- (m-1-2);
    \end{scope}
  \end{tikzpicture}
\end{equation*}
and the isomorphisms $\Cuntz_{\Asf'}(\{x\})\cong\Cuntz_\Asf(\{x\})\cong\Cuntz_{\Asf_1}$.

Mapping the class of the extension~$\Ecal_\Nsf$ towards the bottom, we obtain the classes
$\bigl[\Cuntz_{\Asf}(U_x\setminus\{x\})\rightarrowtail\Cuntz_{\Asf}(U_x)\twoheadrightarrow\Cuntz_{\Asf}(\{x\})\bigr]\in\KK_1\bigl(\Cuntz_{\Asf}(\{x\}),\Cuntz_{\Asf}(U_x\setminus\{x\})\bigr)$ and
\[
\bigl[\Cuntz_{\Asf'}(U_x\setminus\{x\})\rightarrowtail\Cuntz_{\Asf'}\twoheadrightarrow\Cuntz_{\Asf'}(\{x\})\bigr]\in\KK_1\bigl(\Cuntz_{\Asf'}(\{x\}),\Cuntz_{\Asf'}(U_x\setminus\{x\})\bigr).
\]
Translating further along the inverse of the bottom arrow yields precisely the desired class~$\beta$ of the extension $\left(\Comp\otimes\Cuntz_{\Asf_2'}\rightarrowtail\Cuntz_{\Asf'}\twoheadrightarrow\Cuntz_{\Asf_1}\right)$. It follows that the top arrow maps~$\alpha$ to~$\left[\Ecal_\Nsf\right]$.
\end{proof}

Let us complete the proof of Theorem~\ref{thm:no_phantoms}. Using the lemma, we may modify \eqref{eq:morphism_of_triangles} to obtain the following commutative diagram in~$\KKcat(X)$ whose rows are exact triangles:
\begin{equation*}
  \begin{tikzpicture}[baseline=(current bounding box.west)]
    \matrix(m)[cd,column sep=4em]{
      \Sigma P_{\{x\}}\Acal&P_{X\setminus\{x\}}\Acal&\Acal&P_{\{x\}}\Acal\\
      \Sigma P_{\{x\}}\Cuntz_{\Asf}&P_{X\setminus\{x\}}\Cuntz_{\Asf}&\Cuntz_\Asf &P_{\{x\}}\Cuntz_{\Asf}.\\
    };
    \begin{scope}[cdar]
      \draw (m-1-1) -- (m-1-2);
      \draw (m-1-2) -- (m-1-3);
      \draw (m-1-3) -- (m-1-4);
      \draw (m-2-1) -- node {\(\left[\Ecal_{\Nsf}\right]\)} (m-2-2);
      \draw (m-2-2) -- (m-2-3);
      \draw (m-2-3) -- (m-2-4);
      \draw (m-1-1) -- node {\(\simeq\)} (m-2-1);
      \draw (m-1-2) -- node {\(\simeq\)}(m-2-2);
      \draw (m-1-4) -- node {\(\simeq\)}(m-2-4);
    \end{scope}
  \end{tikzpicture}
\end{equation*}
It follows from an axiom of triangulated categories and the Five Lemma for triangulated categories that $\Acal$ and~$\Cuntz_\Asf$ are $\KK(X)$-equiv\-a\-lent and thus stably isomorphic by Kirchberg's classification theorem~\cite{Kirchberg:Michael}*{Folgerung~4.3}.

Finally, if $\Acal$ is unital then it is a Cuntz--Krieger algebra by~\cite{Arklint-Ruiz:Corners_of_CK_algs}*{Theorem~4.8(1)}. This concludes the proof of Theorem~\ref{thm:no_phantoms}. \hspace{\stretch1}\qedsymbol

\end{document}